\documentclass{amsart}
\usepackage{dsfont}
\usepackage{mathrsfs}
\usepackage{color}
\usepackage{amssymb}
\newtheorem{theorem}{Theorem}[section]
\newtheorem{lemma}[theorem]{Lemma}
\usepackage[nosort]{cite}
\theoremstyle{definition}
\newtheorem{definition}{Definition}[section]
\newtheorem{example}[theorem]{Example}

\newtheorem{proposition}{Proposition}[section]

\theoremstyle{remark}

\newtheorem{remark}{Remark}[section]

\numberwithin{equation}{section}

\makeatletter

\newcommand{\Rmnum}[1]{\expandafter\@slowromancap\romannumeral #1@}
\makeatother

\begin{document}

\title{ Topological pressure of free semigroup actions for non-compact sets and Bowen's equation, \uppercase\expandafter{\romannumeral2}}


\author{Qian Xiao}
\address{School of Mathematics, South China University of Technology, Guangzhou 510641, P.R. China}
\email{qianqian1309200581@163.com}

\author{Dongkui Ma*}
\thanks{* Corresponding author}
\address{School of Mathematics, South China University of Technology,
Guangzhou 510641, P.R. China}
\email{dkma@scut.edu.cn}





\begin{abstract}
 Inspired to the work of Ma and Wu\cite{Ma} and Climenhaga\cite{Climenhaga}, we introduce the new nation of topological pressure of a semigroup of maps by using the  Carath\'{e}odory-Pesin structure (C-P structure) with respect to arbitrary subset in this paper. Moreover,  by Bowen's equation, we characterize the Hausdorff dimension of an arbitrary subset, where the points of the subset have the positive lower Lyapunov exponents and satisfy a so called tempered contraction condition.
\end{abstract}

\maketitle

\section{ \emph{Introduction}}
 Topological pressure, as a natural extension of topological entropy, plays a fundamentable role in dynamical systems.
 Ruelle \cite{Ruelle} and Walters \cite{Walters} brought the concept of topological pressure and it was further developed by Pesin and Pitskel \cite{Pesin,Pesin1}.
Pesin \cite{Pesin} use the ``dimension" approach to the notion of topological pressure which is based on the Carath\'{e}odory structure.
which we called it Carath\'{e}odory-Pesin structure (namely, C-P structure). It is a very powerful tool to study dimension theory and dynamical systems.

Bowen's equation plays a key role in dynamical systems, which is a bridge between fractals and dynamical systems.
For a long time, people have done a lot of research on Bowen's equation.
Bowen\cite{Bowen3} first investigated the relationship between topological pressure and Hausdorff dimension. He showed that, for a certain compact set $J \subset \mathbb{C}$ which is invariant set of fractional linear transformations $f$ of the Riemann sphere, the Hausdorff dimension $t=dim_{H} J$ is the unique root of the equation (namely, Bowen's equation)
\begin{equation}
P_{J}(-t\varphi)=0,
\end{equation}
where $P_{J}$ is the topological pressure of the map $f: J \rightarrow J$, and $\varphi(z)=\log |f'(z)|$.
When $f$ is a $C^{1+\varepsilon}$ conformal map on a Riemann manifold and $J$ is a repeller, Ruelle\cite{Ruelle} showed that Bowen's equation(1.1) yields the Hausdorff dimension of $J$.
Gatzouras and Peres\cite{Gatzouras} further extended the result to the case where $f$ is $C^{1}$.
Rugh\cite{Rugh} redefined the conformal map and confirmed that the result remains valid as $X$ is a metric space (not necessarily a manifold).
Climenhaga \cite{Climenhaga} characterized the relationship between the Hausdorff dimension of certain sets $Z$ and topological pressure on $Z$, whose point has positive lower Lyapunov exponents and satisfies tempered contraction condition.
 By the Bowen's equation, there exists extensive literature that studied the settings of Julia sets and non-conformal repellers (see, for example,\cite{Ban,Cao,Denker,Mayer,Mayer1,Przytycki,Przytycki1,Urbanski1,Urbanski2,Urbanski3}).

The above related results are focused on a single maps. Later, Bufetov \cite{Bufetov}and Bi\'{s}\cite{Bis} respectively proposed two kinds of topological entropies to a semigroup action generated by finite continuous maps.
Moreover, Ma and Wu\cite{Ma} adopt C-P structure to derive the idea of topological entropy and topological pressure on any subsets, which generalized the definition of Bi\'{s}\cite{Bis}.
Later, Lin et al\cite{Lin} provided the notion of topological pressure for a free semigroup action on a compact metric space, which is extended the result of  Bufetov \cite{Bufetov}.
Then Ju et al. \cite{Ju} extended this concept to arbitrary subsets by using C-P structure and gave some properties and applications.
On the basis of the notion of Ju, Ma and Wang\cite{Ju}, Xiao and Ma \cite{Xiao} introduced the concept of topological pressure of free semigroup actions on arbitrary subsets. Furthermore, inspired by the results of Climenhaga\cite{Climenhaga}, Xiao and Ma \cite{Xiao}  gave the relationship between the topological pressure on $Z \subset X$ and the Hausdorff  dimension of a certain subset $Z$, whose points have positive lower Lyapunov exponents and satisfy a tempered contraction condition.

Motivated by Ma and Wu \cite{Ma}, we give the new notion of topological pressure of a semigroup of continuous maps with finite functions on arbitrary subset of $X$ by using C-P structure.
Meanwhile, we also obtain some properties of it.
Furthermore, in the meaning of the new topological pressure defined in this paper, we give the relationship between the Hausdorff dimension of a certain subset $Z$ and the topological pressure on $Z \subset X$ by Bowen equation.

The paper is organized as follows. In section 2, we review some preliminaries. In section 3, we give the new definitions of the topological pressure and lower and upper capacity topological pressures of a semigroup by using the C-P structure. Moreover, the equivalent definition is given in an alternative way and several of their properties are provided. In section 4, we show the Theorem 4.1.

\section{\emph{Preliminaries}}

Let $(X,d)$ be a compact metric space. Consider a semigroup $G$ of continuous transformations of $X$ into itself.
The semigroup $G$ is assumed to be finitely generated, namely, there exists a finite set $G_{1}=\{id_{X},f_{0},f_{1},\cdots,f_{k-1}\}$ such that $G=\bigcup_{n \in \mathbb{N}}G_{n}$, where $G_{n}=\{g_{1}\circ\cdots \circ g_{n}:g_{1},\cdots,g_{n} \in G_{1}\}$.
Obviously, $G_{m} \subset G_{n}$ for all $m \leq n$.

Let $F_{k}^+$ be the set of all finite words of alphabet $\{0,1,2,\ldots,k-1\}$. For any $w \in F_{k}^+$, $ |w|$ stands for the length of $w$, that is, the number of symbols in $w$.
Denote $F_{k}^+(n)=\{w\in F_{k}^+: |w|=n, n \in \mathbb{N}\}$.
Obviously, $F_{k}^+$ with respect to the law of composition is a free semigroup with $k$ generators.
We write ${w}^{\prime}\leq{w}$ if there exists a word $w^{\prime\prime}\in F_{k}^+$ such that $w=w^{\prime\prime}w^{\prime}$.
For $w=i_1\ldots i_n \in F_{k}^{+}$, denote $\overline{w}=i_n\ldots i_1$.
Given a real number $\delta > 0$, ~$w\in F_{k}^+$ and a point $x\in X$, define the $(w,\delta)$-Bowen ball at $x$ by
\[
B_{w}(x,\delta)=\{y \in X : d(f_{w'}(x),f_{w'}(y)) < \delta, ~{\rm for} ~w' \le \overline{w} \}.
\]

Let $ (X,d) $ be a metric space and $\mathcal{P}(Z,r)$ denote the collection of countable open covers $\{U_{i}\}_{i=1}^{\infty}$ of $Z$ for which $diam U_{i} < r$ for all $i$. Given a subset $ Z \subset X, t\geq 0$, define
\[
\mathcal{H}(Z,t, r)=\inf_{\mathcal{P}(Z,r)}\left\{\sum\limits_{U_{i} \in \mathcal{P}(Z,r)} ({\rm diam}(U_{i}))^t \right\}.
\]
As $r$ decreases, $\mathcal{H}(Z,t, r)$ increases. Therefore, there exists the limit
\[
\mathcal{H}(Z,t)=\lim\limits_{r\rightarrow 0} \mathcal{H}(Z,t, r).
\]
We call $\mathcal{H}(Z,t)$ the \emph {t-dimensional Hausdorff measure} of $Z$.
Moreover, there is a critical value of $t$ at which $\mathcal{H}(Z,t)$ jumps from $\infty$ to $0$, which is called the \emph {Hausdorff dimension} of $Z$. For details, see \cite{Falconer}. Formally,
\[
dim_{H}(Z)=\inf\{t >0: \mathcal{H}(Z,t)=0\}=\sup\{t>0 : \mathcal{H}(Z,t)=\infty\}.
\]

One may equivalently define Hausdorff dimension of using covers by open balls rather than arbitrary open sets. Let $\mathcal{P}^{b}(Z,r)$ denote the collection of countable open balls covers $\{B(x_{i},r_{i})\}_{i=1}^{\infty}$ of $Z$ with $x_{i}\in Z$ and $r_{i} < r$ for all $i$, and then define
\[
\mathcal{H}^{b}(Z,t,r)=\inf_{\mathcal{P}^{b}(Z,r)}\left\{\sum\limits_{B(x_{i},r_{i}) \in \mathcal{P}^{b}(Z,r)} ({\rm diam} B(x_{i},r_{i}))^t\right\}.
\]
Finally, using the same procedure as above we can also define $\mathcal{H}^{b}(Z,t)$ and $dim_{H}^{b}Z$. Furthermore, we have known $dim_{H}^{b}Z=dim_{H}Z$ from \cite{Climenhaga, Falconer}.

\section{\emph{Topological pressure and lower and upper capacity topological pressures of a free semigroup action and their properties}}

\subsection{Topological pressure and lower and upper capacity topological pressures}

Let $X$ be a compact metric space with metric $d$.
Given $\varphi_{0},\cdots,\varphi_{k-1} \in C(X,\mathbb{R})$, denote $\Phi=\{\varphi_{0},\cdots,\varphi_{k-1} \}$.
Let $w=i_1i_2 \ldots i_n \in F_{k}^+$ and $f_w=f_{i_1}\circ f_{i_2}\circ \ldots \circ f_{i_n}$, where $i_j=2,\ldots,k$ for all $j=1,\ldots,n$ . Obviously, $f_{ww'}=f_wf_{w'}$.
For $w=i_{1}i_{2}\cdots i_{n} \in F_{k}^{+}$, denote
\[
S_{w}\Phi(x):=\varphi_{i_{1}}(x)+\varphi_{i_{2}}(f_{i_{1}}(x))+\cdots+\varphi_{ i_{n}}(f_{i_{n-1}i_{n-2}\cdots i_{1}}(x)).
\]

Define a new metric $d_{n}$ on $X$ by
\[
d_{n}(x,y)=\max \left\{d(g(x),g(y): g \in G_{n}\right\}.
 \]
Clearly, if $n_{1} \leq n_{2}$, then $d_{n_{1}}(x_1,x_2) \leq d_{n_{2}}(x_1,x_2)$ for all $x_1,x_2 \in X$.

Considering a real number $\delta > 0$ and a point $x\in X$, we can define the $(n,\delta)$-Bowen ball at $x$ by
\[
B_{n}(x,\delta)=\{y \in X : d_{n}(x,y) < \delta \}.
\]

Given $ Z\subset X, \alpha \in \mathbb{R}$ and $N>0$, define
\begin{align*}
M(Z,G_{1},\Phi,\alpha,\delta,N):
=\inf\limits_{\mathcal{G} }\left\{\sum\limits_{B_{n}(x,\delta)\in\mathcal{G} }\exp\left(-\alpha\cdot n+\sup_{y \in B_{n}(x,\delta)} \frac{1}{k^{n}} \sum_{|w|=n}S_{w}\Phi(y) \right)\right\},
\end{align*}
where the infimum runs over all finite or countable subcollections $\mathcal{G}=\{B_{n}(x,\delta)\}$ such that $x \in X, ~n\geq N$ and $\bigcup\limits_{B_{n}(x,\delta) \in \mathcal{G}} B_{n}(x,\delta) \supset Z.$

We can easily verify that the function $M(Z,G_{1},\Phi,\alpha,\delta,N)$ is non-decreasing as $N$ increases.
Therefore, there exists the limit
\[
m(Z,G_{1},\Phi,\alpha,\delta)=\lim\limits_{N\to\infty}M(Z,G_{1},\Phi,\alpha,\delta,N).
\]

Moreover, we can also define
\begin{align*}
R(Z,G_{1},\Phi,\alpha,\delta,N)
&=\inf\limits_{\mathcal{G}_{N}}\left\{\sum\limits_{B_{N}(x,\delta) \in \mathcal{G}_{N}}\exp\left(-\alpha\cdot N+\sup_{y \in B_{N}(x,\delta)} \frac{1}{k^{N}} \sum_{|w|=N}S_{w}\Phi(y) \right)\right\},\\
\end{align*}
where the infimum is taken over all finite or countable subcollections $\mathcal{G}_{N}$ covering $Z$( i.e., $\bigcup\limits_{B_{N}(x,\delta) \in \mathcal{G}_{N}} B_{N}(x,\delta) \supset Z$).

Let
\[
\underline{r}(Z,G_{1},\Phi,\alpha,\delta)=\liminf_{N\to\infty}R(Z,G_{1},\Phi,\alpha,\delta,N),
\]
\[
\overline{r}(Z,G_{1},\Phi,\alpha,\delta)=\limsup_{N\to\infty}R(Z,G_{1},\Phi,\alpha,\delta,N).
\]

By the construction of C-P structure (for details see \cite{Pesin}), there exist unique critical values respectively such that the quantities $m(Z,G_{1},\Phi,\alpha,\delta)$, $\underline{r}(Z,G_{1},\Phi,\alpha,\delta)$,\\
$\overline{r}(Z,G_{1},\Phi,\alpha,\delta)$ jump from $\infty$ to $0$. We denote them as $P_{Z}(G_{1},\Phi,\delta), \underline{CP}_{Z}(G_{1},\Phi,\delta),$ and $\overline{CP}_{Z}(G_{1},\Phi,\delta)$ respectively. Accordingly, we have that
\[
P_{Z}(G_{1},\Phi,\delta)=\inf\{\alpha: m(Z,G_{1},\Phi,\alpha,\delta)=0\}=\sup\{\alpha: m(Z,G_{1},\Phi,\alpha,\delta)=\infty\},
\]
\[
\underline{CP}_{Z}(G_{1},\Phi,\delta)=\inf\{\alpha: \underline{r}(Z,G_{1},\Phi,\alpha,\delta)=0\}=\sup\{\alpha: \underline{r}(Z,G_{1},\Phi,\alpha,\delta)=\infty\},
\]
\[
\overline{CP}_{Z}(G_{1},\Phi,\delta)=\inf\{\alpha: \overline{r}(Z,G_{1},\Phi,\alpha,\delta)=0\}=\sup\{\alpha: \overline{r}(Z,G_{1},\Phi,\alpha,\delta)=\infty\}.
\]

Similar to the argument of Pesin\cite{Pesin} and Ma and Wu \cite{Ma}, the limits $\lim\limits_{\delta\rightarrow 0}P_{Z}(G_{1},\Phi,\delta)$, $\lim\limits_{\delta\rightarrow 0}\underline{CP}_{Z}(G_{1}$,$\Phi,\delta),\lim\limits_{\delta\rightarrow 0}\overline{CP}_{Z}(G_{1},\Phi,\delta)$ exist respectively.

\begin{definition}\label{4.1t}
For any set $Z \subset X$, we call the following quantities:
 \[
 P_{Z}(G_{1},\Phi):=\lim\limits_{\delta\rightarrow 0}P_{Z}(G_{1},\Phi,\delta),
 \]
 \[
 \underline{CP}_{Z}(G_{1},\Phi):=\lim\limits_{\delta\rightarrow 0}\underline{CP}_{Z}(G_{1},\Phi,\delta),
 \]
 \[
 \overline{CP}_{Z}(G_{1},\Phi):=\lim\limits_{\delta\rightarrow 0}\overline{CP}_{Z}(G_{1},\Phi,\delta),
 \]
 the topological pressure and lower and upper capacity topological pressures on the set $Z$ respectively.
\end{definition}

\begin{remark}
 (1) It is easy to see that $P_{Z}(G_{1},\Phi)\le \underline{CP}_{Z}(G_{1},\Phi)\le \overline{CP}_{Z}(G_{1},\Phi)$.

 (2) If $\varphi_{0}=\cdots=\varphi_{k-1}=\varphi$, it is easy to verify $P_{Z}(G_{1},\Phi)\leq P_{Z}(G_{1},\varphi)$, where $P_{Z}(G_{1},\varphi)$ is defined by Ma and Wu \cite{Ma}.

 (3) When $\varphi_{0}=\cdots=\varphi_{k-1}=0$, it is the topological entropy $h_{Z}(G_{1})$ defined by Ma and Wu \cite{Ma}.
\end{remark}

\subsection{Equivalent definition by using the center of Bowen's ball}\

Given $ Z\subset X, \alpha \in \mathbb{R}$ and $N>0$, we can define new functions as follows
\[
M'(Z,G_{1},\Phi,\alpha,\delta,N)=\inf\limits_{\mathcal{G} }\left\{\sum\limits_{B_{n}(x,\delta)\in\mathcal{G} }\exp\left(-\alpha\cdot n+ \frac{1}{k^{n}} \sum_{|w|=n}S_{w}\Phi(x) \right)\right\}
\]
where the infimum runs over all finite or countable subcollections $\mathcal{G}=\{B_{n}(x,\delta): x \in Z, n \geq N\}$ such that $\bigcup\limits_{B_{n}(x,\delta) \in \mathcal{G}} B_{n}(x,\delta) \supset Z,$
 and
\[
R'(Z,G_{1},\Phi,\alpha,\delta,N)=\inf\limits_{\mathcal{G}_{N}}\left\{\sum\limits_{B_{N}(x,\delta) \in \mathcal{G}_{N}}\exp\left(-\alpha\cdot N+ \frac{1}{k^{N}} \sum_{|w|=N}S_{w}\Phi(x) \right)\right\},
\]
where the infimum runs over all finite or countable subcollections $\mathcal{G}_{N}=\{B_{N}(x,\delta): x \in Z\}$ such that $\bigcup\limits_{B_{n}(x,\delta) \in \mathcal{G}} B_{n}(x,\delta) \supset Z.$

By the same procedure as above, we can get $m'(Z,G_{1},\Phi,\alpha,\delta),\  \underline{r'}(Z,G_{1},\Phi,\alpha,\delta)$ and $\overline{r'}(Z,G_{1},\Phi,\alpha,\delta)$ respectively. Finally, for any $Z\subset X$, we denote the respective critical values by
\[
P'_{Z}(G_{1},\Phi,\delta),\ \underline{CP'}_{Z}(G_{1},\Phi,\delta),\ \overline{CP'}_{Z}(G_{1},\Phi,\delta).
\]

\begin{theorem}\label{4.2t}
For any set $Z \subset X$, the following hold:
 \[P_{Z}(G_{1},\Phi)=\lim\limits_{\delta\rightarrow 0}P'_{Z}(G_{1},\Phi,\delta),\]
 \[\underline{CP}_{Z}(G_{1},\Phi)=\lim\limits_{\delta\rightarrow 0}\underline{CP'}_{Z}(G_{1},\Phi,\delta),\]
 \[\overline{CP}_{Z}(G_{1},\Phi)=\lim\limits_{\delta\rightarrow 0}\overline{CP'}_{Z}(G_{1},\Phi,\delta).\]
\end{theorem}

\begin{proof}
Our approach is inspired by the work of Pesin in \cite{Pesin}. Firstly, given $\delta > 0$, let
\[
\varepsilon(\delta):=\sup \{|\varphi_{i}(x)-\varphi_{i}(y)|: d(x,y)< \delta, i=0,1,\cdots,k-1\}.
\]
Owing to the fact that $\varphi_{i} \in C(X,\mathbb{R})$ and $X$ is compact, $\varphi_{i}$ is uniformly continuous. Hence $\varepsilon(\delta)$ is finite and
$\lim\limits_{\delta \rightarrow 0}\varepsilon(\delta)=0$. Moreover, for any $y \in B_{n}(x, \delta),~|w|=n$, we have
\[
\left|S_{w}\Phi(x)-S_{w}\Phi(y)\right| \leq n\varepsilon(\delta).
\]

Consider the collection $\mathcal{G}=\{B_{n}(x,\delta): x \in Z, n\geq N\}$ such that $\bigcup\limits_{B_{n}(x,\delta) \in \mathcal{G}} B_{n}(x,\delta) \supset Z$, then
\begin{align*}
M(Z,G_{1},\Phi,\alpha,\delta,N)
& \leq \sum\limits_{B_{n}(x, \delta)\in\mathcal{G}}\exp\left(-\alpha n+\sup_{y \in B_{n}(x, \delta)} \frac{1}{k^{n}} \sum_{|w|= n}S_{w}\Phi(y) \right)\\
&\leq \sum\limits_{B_{n}(x,\delta) \in\mathcal{G}}\exp\Big( -\alpha n+ \frac{1}{k^{n}} \sum_{|w|=n}S_{w}\Phi(x)+n\varepsilon(\delta) \Big)\\
& = \sum\limits_{B_{n}(x,\delta)\in\mathcal{G}}\exp\Big( -n(\alpha-\varepsilon(\delta))+ \frac{1}{k^{n}} \sum_{|w|=n}S_{w}\Phi(x)\Big).\\
\end{align*}
Thus we can get
\[
M(Z,G_{1},\Phi,\alpha,\delta,N) \leq  M'(Z,G_{1},\Phi,\alpha-\varepsilon(\delta),\delta,N).
\]
Taking the limit $N \rightarrow \infty$ yields
\[
m(Z,G_{1},\Phi,\alpha,\delta) \leq m'(Z,G_{1},\Phi,\alpha-\varepsilon(\delta),\delta).
\]
Therefore,
\[
P_{Z}(G_{1},\Phi,\delta) \leq P'_{Z}(G_{1},\Phi,\delta)+\varepsilon(\delta),
\]
and as $\delta \rightarrow 0$, that is, $\varepsilon(\delta)\rightarrow 0$,  we obtain
\[
P_{Z}(G_{1},\Phi) \leq \liminf_{\delta\rightarrow 0}P'_{Z}(G_{1},\Phi,\delta).
\]

Next, choosing $\mathcal{G}=\{B_{n}(x,\delta/2): x \in X, n\geq N\}$ such that $\bigcup\limits_{B_{n}(x,\delta/2) \in \mathcal{G}} B_{n}(x,\delta/2) \supset Z$, we may assume without loss of generality that for every $B_{n}(x,\delta/2) \in \mathcal{G}$, we have $B_{n}(x,\delta/2) \cap Z\neq \emptyset.$ Thus, for each such $B_{n}(x,\delta/2)$, we choose $y \in B_{n}(x,\delta/2) \cap Z$; we see that $B_{n}(x,\delta/2) \subset B_{n}(y,\delta).$ Set $\mathcal{G}'=\{B_{n}(y, \delta): y \in B_{n}(x,\delta/2) \cap Z\}$, then $\bigcup\limits _{B_{n}(y, \delta) \in \mathcal{G}'}B_{n}(y, \delta) \supset Z$ and we can get
\begin{align*}
 M(Z,G_{1},\Phi,&\alpha,\delta/2,N) \\
& =\inf\limits_{\mathcal{G} }\left\{\sum\limits_{B_{n}(x,\delta/2)\in\mathcal{G} }\exp\left(-\alpha\cdot n+\sup_{z \in B_{n}(x,\delta/2)} \frac{1}{k^{n}} \sum_{|w|=n}S_{w}\Phi(z) \right)\right\}\\
 &\geq \inf\limits_{\mathcal{G}' }\left\{\sum\limits_{B_{n}(y,\delta)\in\mathcal{G}' }\exp\left(-\alpha\cdot n+\frac{1}{k^{n}} \sum_{|w|=n}S_{w}\Phi(y)-n \varepsilon(\delta) \right)\right\}\\
  & \geq M'(Z,G_{1},\Phi,\alpha+\varepsilon(\delta),\delta,N).
\end{align*}
Hence
\[
m(Z,G_{1},\Phi,\alpha,\delta/2) \geq   m'(Z,G_{1},\Phi,\alpha+\varepsilon(\delta),\delta).
\]
As a result,
\[
P_{Z}(G_{1},\Phi,\delta/2) \geq P'_{Z}(G_{1},\Phi,\delta)-\varepsilon(\delta),
\]
and taking the limit $\delta \rightarrow 0$ gives
\[
P_{Z}(G,\Phi) \geq \limsup_{\delta\rightarrow 0}P'_{Z}(G,\Phi,\delta),
\]
which finishes the proof of the first.
The existence of the two other limits can be proved in a similar way.
\end{proof}

\subsection{Properties of topological pressure and lower and upper capacity topological pressures}\

The following basic properties of topological pressure and lower and upper capacity topological pressures of a semigroup can be verified directly from the basic properties of the Carath\'{e}odory-Pesin  dimension \cite{Pesin} and definitions.

\begin{proposition}\label{3.1p}
(1) $P_{\emptyset}(G_{1},\Phi)\le 0.$

(2) $P_{Z_{1}}(G_{1},\Phi)\le P_{Z_{2}}(G_{1},\Phi)$, ~ if~ $Z_{1}\subset Z_{2}$.

(3) $P_{Z}(G_{1},\Phi)=\sup\limits_{i\ge1}P_{Z_{i}}(G_{1},\Phi)$,~ where~$Z=\cup_{i\ge1}Z_{i}$ ~and~$Z_{i}\subset X , i=1,2,\cdots $.
\end{proposition}

\begin{proposition}\label{3.2p}
(1) $\underline{CP}_{\emptyset}(G_{1},\Phi)\le 0,~\overline{CP}_{\emptyset}(G_{1},\Phi)\le 0.$

(2) $\underline{CP}_{Z_{1}}(G_{1},\Phi)\le \underline{CP}_{Z_{2}}(G_{1},\Phi)$ ~and~ $\overline{CP}_{Z_{1}}(G_{1},\Phi)\le  \overline{CP}_{Z_{2}}(G_{1},\Phi)$, ~if~ $Z_{1}\subset Z_{2}$.

(3) $\underline{CP}_{Z}(G_{1},\Phi) \ge \sup\limits_{i\ge1}\underline{CP}_{Z_{i}}(G_{1},\Phi)$ and $\overline{CP}_{Z}(G_{1},\Phi) \ge \sup\limits_{i\ge1}\overline{CP}_{Z_{i}}(G_{1},\Phi)$,
where $Z=\cup_{i\ge1}Z_{i}$ ~and~$Z_{i}\subset X , i=1,2,\cdots $.

(4) If $g:X \rightarrow X$ is a homeomorphism which commutes with $G_{1}$ (i.e., $f_{i}\circ g=g \circ f_{i}$,
for all $f_{i}\in G_{1}$), then
\[
P_{Z}(G_{1},\Phi)=P_{g(Z)}(G_{1},\Phi \circ g^{-1}),
\]
\[
\underline{CP}_{Z}(G_{1},\Phi)=\underline{CP}_{g(Z)}(G_{1},\Phi \circ g^{-1}),
\]
\[
\overline{CP}_{Z}(G_{1},\Phi)=\overline{CP}_{g(Z)}(G_{1},\Phi \circ g^{-1}),
\]
where $\Phi \circ g^{-1}=\{\varphi_{0}\circ g^{-1},\cdots,\varphi_{k-1} \circ g^{-1}\}$.
\end{proposition}

We also point out the \emph{continuity property of topological pressure and lower and upper capacity topological pressures.}
\begin{theorem}\label{3.3t}
If $\Phi=\{\varphi_{0},\cdots, \varphi_{k-1}\}$ and $\Psi=\{\psi_{0},\cdots, \psi_{k-1}\}$,  we have
\[
|P_{Z}(G_{1},\Phi)-P_{Z}(G_{1},\Psi)| \leq \max_{0 \leq i \leq k-1} \|\varphi_{i}-\psi_{i}\|,
\]
\[
|\underline{CP}_{Z}(G_{1},\Phi)-\underline{CP}_{Z}(G_{1},\Psi)| \leq \max_{0 \leq i \leq k-1} \|\varphi_{i}-\psi_{i}\|,
\]
\[
|\overline{CP}_{Z}(G_{1},\Phi)-\overline{CP}_{Z}(G_{1},\Psi)| \leq \max_{0 \leq i \leq k-1} \|\varphi_{i}-\psi_{i}\|,
\]
where $\|\cdot\|$ denotes the supremum norm in the space of continuous functions on X.
\end{theorem}

\begin{proof}
Note that for any $w\in F_{k}^+, ~x\in X$,
 \[
 S_{w}\Phi(x)-S_{w}\Psi(x) \leq |w| \max_{0 \leq i \leq k-1} \|\varphi_{i}-\psi_{i}\|,
 \]
 and
 \[
\frac{1}{k^{n}} \sum_{|w|=n}S_{w}\Phi(x)-\frac{1}{k^{n}} \sum_{|w|=n}S_{w}\Psi(x) \leq n \max_{0 \leq i \leq k-1} \|\varphi_{i}-\psi_{i}\|.
 \]
Therefore,
\begin{align*}
 M'(Z,G_{1},\Phi,\alpha,\delta,N)
 \leq  M'(Z,G_{1},\Psi,\alpha-\max_{0 \leq i \leq k-1}\|\varphi_{i}-\psi_{i}\|,\delta,N).
 \end{align*}
Thus,
 \begin{align*}
M'(Z,G_{1},\Phi,\alpha,\delta,N) \leq   M'(Z,G_{1},\Psi,\alpha-\max_{0 \leq i \leq k-1}\|\varphi_{i}-\psi_{i}\|,\delta,N).
 \end{align*}
Taking limit $N \rightarrow \infty$ yields
\begin{align*}
 m'(Z,G_{1},\Phi,\alpha,\delta) \leq  m'(Z,G_{1},\Psi,\alpha-\max_{0 \leq i \leq k-1}\|\varphi_{i}-\psi_{i}\|,\delta).
\end{align*}
Therefore,
\[
P'_{Z}(G_{1},\Phi,\delta) \leq P'_{Z}(G_{1},\Psi,\delta)+\max_{0 \leq i \leq k-1}\|\varphi_{i}-\psi_{i}\|.
\]
Similarly, we can also get
\[
P'_{Z}(G_{1},\Phi,\delta) \geq P'_{Z}(G_{1},\Psi,\delta)-\max_{0 \leq i \leq k-1}\|\varphi_{i}-\psi_{i}\|.
\]
Let $\delta\rightarrow 0$, and we obtain
\[
P_{Z}(G_{1},\Psi)-\max_{0 \leq i \leq k-1}\|\varphi_{i}-\psi_{i}\| \leq P_{Z}(G_{1},\Phi) \leq P_{Z}(G_{1},\Psi)+\max_{0 \leq i \leq k-1}\|\varphi_{i}-\psi_{i}\|,
\]
which establishes the first inequality.
\end{proof}

\section{\emph{Bowen's equation}}

In this section, using the topological pressure defined in Section 3.1, we give a qualitative characterization of the Hausdoff dimension of the set $Z$, whose points have the positive lower Lyapunov exponents and satisfy a tempered contraction condition.

\begin{definition}\label{7.1d}
Let $(X,d)$ be a compact metric space and $f: X \longrightarrow X$ a continuous map. A continuous map $f: X \longrightarrow X$ is called \emph{conformal} with factor $a(x)$ if for every $x \in X$, we have
\[
a(x)=\lim_{y\rightarrow x}\frac{d(f(x),f(y))}{d(x,y)},
\]
where $a:X\longrightarrow [0,\infty)$ is continuous.
\end{definition}

Let $(X,d)$ be a compact metric space and $G$ be a semigroup of continuous transformations of $X$ generated by $G_{1}$,where $G_{1}=\{id_{X},f_{0},\cdots,f_{k-1}\}$ and $f_0,\ldots,f_{k-1}$  are conformal with factor $a_{i}(x),~i=0,\cdots,k-1$. Set $\Phi=\{\log a_{0},\cdots, \log a_{k-1}\}$.

For $w=i_{1}i_{2}\cdots i_{n} \in F_{k}^{+}$, denote
\[
S_{w}\Phi(x):=\log a_{i_{1}}(x)+\log a_{i_{2}}(f_{i_{1}}(x))+\cdots+\log a_{i_{n}}(f_{i_{n-1}i_{n-2}\cdots i_{1}}(x))
\]
and
\begin{align*}
\lambda_{w}(x)&=\frac{1}{|w|}S_{w}\Phi(x),
\end{align*}
where
\[
f_{i_{n-1}i_{n-2}\cdots i_{1}}:=f_{i_{n-1}}\circ f_{i_{n-2}} \circ \cdots \circ f_{i_{1}}.
\]
Let

\begin{align*}
\lambda_{n}(x):&=\frac{1}{n}\left(\frac{1}{k^{n}} \sum_{|w|=n}S_{w}\Phi(x)\right)\\
&=\frac{1}{n}\Bigg(\frac{\sum_{i_{1}\in F_{k}^{+}(1)}\varphi_{i_{1}}(x)}{k}+\frac{\sum_{i_{1}i_{2}\in F_{k}^{+}(2)}\varphi_{i_{2}}(f_{i_{1}}(x))}{k^{2}}+\cdots+\\
&\frac{\sum_{i_{1}i_{2}\cdots i_{n-1}\in F_{k}^{+}(n-1)}\varphi_{i_{n-1}}(f_{i_{n-2}\cdots i_{1}}(x))}{k^{n-1}}+\frac{\sum_{i_{1}i_{2}\cdots i_{n}\in F_{k}^{+}(n)}\varphi_{i_{n}}(f_{i_{n-1}\cdots i_{1}}(x))}{k^{n}}\Bigg).
\end{align*}
 It is easy to obtain
\[
\lambda_{n}(x)=\frac{1}{k^{n}} \sum_{|w|=n} \lambda_{w}(x).
\]
Set
\[
\underline{\lambda}(x)=\liminf_{n\rightarrow \infty} \lambda_{n}(x),
\]
and
\[
\overline{\lambda}(x)=\limsup_{n\rightarrow \infty} \lambda_{n}(x),
\]
then we call $\underline{\lambda}(x)$ and $\overline{\lambda}(x)$ the \emph{lower and upper Lyapunov exponents} of the semigroup $G$ at $x$. If the two are equal, their common value is called the \emph{Lyapunov exponent} at $x$:
\[
\lambda(x)=\lim_{n\rightarrow \infty} \lambda_{n}(x).
\]

Given $E \subset \mathbb{R}$, let $\mathcal{A}(E)$ be the set of points along whose orbits all the asymptotic exponential expansion rates of the $G$ lie in $E$:
\[
\mathcal{A}(E)=\{x\in X:[\underline{\lambda}(x),\overline{\lambda}(x)]\subseteq E\}.
\]
In particular, $\mathcal{A}((0,\infty))$ is the set of all points for which $\underline{\lambda}(x)>0$ and $\mathcal{A}(\alpha)=\mathcal{A}(\{\alpha\})$.

In \cite{Climenhaga}, Climenhaga introduced the so-called tempered contraction condition. On the base of \cite{Climenhaga}, we propose the  following so-called \emph{tempered contraction condition}:
\begin{equation}\label{7.2}
\inf_{\substack{n\in \mathbb{N}\\
 w'\in F_{k}^{+}(m) \\
 0< m \leq n} }\bigg\{ \frac{1}{k^{n}}\sum_{|w|=n}S_{w}\Phi(x)-S_{w'}\Phi(x)+n\varepsilon\bigg\}>-\infty,~\text{for any}~\varepsilon>0.
\end{equation}
Denote $\mathcal{B}$ as the set of all points in $X$ which satisfy (\ref{7.2}) .

\begin{example}
Suppose $X:=[0,1]$ and $f_{0}=2x~( mod ~1), f_{1}=2x+1/2~( mod ~1)$, then $G_{1}=\{id_{X}, f_{0}, f_{1}\}, \Phi=\{\log a_{0}, \log a_{1}\}=\{\log 2, \log 2\}$. Thus, $\mathcal{B}=X$.
\end{example}

\begin{theorem}\label{5.3t}
Let $(X,d)$ be a compact metric space and $G$ be a semigroup of continuous transformations of $X$ generated by $G_{1}$,where $G_{1}=\{id_{X},f_{0},\cdots,f_{k-1}\}$ and $f_0,\ldots,f_{k-1}$  are conformal with factor $a_{i}(x),~i=0,\cdots,k-1$.
Suppose that $f_{i}$ has no critical points and singularities, that is, $0< a_{i}(x)< \infty$ for all $x \in X$ and $i \in \{0,\cdots,k-1\}$.
Consider $Z \subset \mathcal{A}((0,\infty))\bigcap \mathcal{B}$ and $\Phi=\{\log a_{0},\cdots, \log a_{k-1}\}$.
Then the Hausdorff dimension of $Z$ is given by
\begin{align*}
dim_{H} Z=t^{*}&=\sup\{ t\geq 0: P_{Z}(G_{1},-t \Phi )>0\}\\
&=\inf \{ t\geq 0: P_{Z}(G_{1},-t \Phi ) \leq 0\}.
\end{align*}
Furthermore, if $Z \subset \mathcal{A}((\alpha,\infty))\bigcap \mathcal{B}$ for some $\alpha >0$, then $t^{*}$ is the unique root of Bowen's equation
\[
P_{Z}(G_{1},-t \Phi )=0.
\]
Finally, if $Z \subset \mathcal{A}(\alpha)$ for some $\alpha >0$, then $P_{Z}(G_{1},-t \Phi )=h_{Z}(G_{1})-t\alpha$, and hence
\[
dim_{H} Z=\frac{1}{\alpha} h_{Z}(G_{1}).
\]
Here $h_{Z}(G_{1})$ is the topological entropy on $Z$ defined by Ma and Wu in \cite{Ma} and $ P_{Z}(G_{1},-t \Phi )$ denotes the topological pressure on $Z$ (see Section 3.1).
\end{theorem}

Before proving the Theorem 4.1, we first give the relevant properties and auxiliary lemmas.

\begin{proposition}\label{7.1p}
Let $f_{i}: X \longrightarrow X$ be as in Theorem 4.2.
Fix $0 < \alpha \leq \beta < \infty$ and $Z \subset \mathcal{A}([\alpha,\beta])$, then

(1) for any $t \in \mathbb{R}$ and $h > 0$, we have
\begin{equation}\label{7.1}
 P_{Z}(G_{1},-t\Phi)-\beta h \leq P_{Z}(G_{1},-(t+h)\Phi)\leq P_{Z}(G_{1},-t\Phi)-\alpha h.
\end{equation}

(2) The equation $P_{Z}(G_{1},-t\Phi)=0$ has unique root $t^{*}$ and
\[
\frac{h_{Z}(G_{1})}{\beta} \leq t^{*} \leq \frac{h_{Z}(G_{1})}{\alpha}.
\]

(3) If $\alpha=\beta$, then $ t^{*} = \frac{h_{Z}(G_{1})}{\alpha}$.\\
Here $h_{Z}(G_{1})$ is the topological entropy defined by Ma and Wu \cite{Ma}, $\Phi=\{\log a_{0},\cdots, \log a_{k-1}\}$ and $t \Phi=\{ t\cdot \log a_{0},\cdots, t\cdot \log a_{k-1}\}$.

\end{proposition}

\begin{proof}
(1) For arbitrary $\varepsilon >0$ and $m \geq 1$, consider
\[
Z_{m}=\Big\{x \in Z: \lambda_{n}(x) \in (\alpha-\varepsilon,\beta+\varepsilon), ~\text{for any}~n\geq m\Big\},
\]
and it is easy to verify  $Z= \cup_{m=1}^{\infty} Z_{m}$. Now fix $t \in \mathbb{R},~h > 0, ~m>0$ and $N \geq m$, then for any $\delta >0$, $p \in \mathbb{R}$, we can get
\begin{align*}
M'(Z_{m},&G_{1},-(t+h)\Phi,p,\delta,N)\\
&=\inf\limits_{\mathcal{G} }\left\{\sum\limits_{B_{n}(x,\delta)\in\mathcal{G} }\exp \bigg( -p \cdot n-(t+h) \frac{1}{k^{n}} \sum_{|w|=n}S_{w}\Phi(x) \bigg)\right\}  \\
&\leq \inf\limits_{\mathcal{G} }\left\{\sum\limits_{B_{n}(x,\delta)\in\mathcal{G}}\exp \bigg(-p\cdot n-t  \frac{1}{k^{n}} \sum_{|w|=n}S_{w}\Phi(x) -h\cdot n(\alpha-\varepsilon)\bigg)\right\} \\
&=\inf\limits_{\mathcal{G} }\left\{\sum\limits_{B_{n}(x,\delta)\in\mathcal{G}}\exp \bigg(-(p+h(\alpha-\varepsilon)\cdot n-t  \frac{1}{k^{n}} \sum_{|w|=n}S_{w}\Phi(x) \bigg)\right\} \\
&=M'(Z_{m},G_{1},-t\Phi,p+h(\alpha-\varepsilon),\delta,N),
\end{align*}
where $\mathcal{G}$ covers $Z_{m}$ and for any $B_{n}(x,\delta) \in \mathcal{G}$, ~$ n \geq N$.
Then
\[
m'(Z_{m},G_{1},-(t+h)\Phi,p,\delta) \leq m'(Z_{m},G_{1},-t\Phi,p+h(\alpha-\varepsilon),\delta).
\]
It follows that
\[
P'_{Z_{m}}(G_{1},-(t+h)\Phi,\delta) \leq P'_{Z_{m}}(G_{1},-t\Phi,\delta)-h(\alpha-\varepsilon).
\]
Letting $\delta \rightarrow 0$, we obtain
\[
P_{Z_{m}}(G_{1},-(t+h)\Phi) \leq P_{Z_{m}}(G_{1},-t\Phi)-h(\alpha-\varepsilon).
\]
Taking the supremum of $m\geq 1$ and by the Proposition \ref{3.1p}, we can get
\[
P_{Z}(G_{1},-(t+h)\Phi) \leq P_{Z}(G_{1},-t\Phi)-h(\alpha-\varepsilon).
\]
Since $\varepsilon > 0$ is arbitrary, the right of inequality (\ref{7.1}) is established.

Using the similar calculation in this manner, we also get
\begin{align*}
M'(Z_{m},G_{1},-(t+h)\Phi,p,\delta,N)
 \geq M'(Z_{m},G_{1},-t\Phi,p+h(\beta+\varepsilon),\delta,N).
\end{align*}
It follows that
\begin{align*}
m'(Z_{m},G_{1},-(t+h)\Phi,p,\delta)
 \geq m'(Z_{m},G_{1},-t\Phi,p+h(\beta+\varepsilon),\delta).
\end{align*}
Hence,
\[
P'_{Z_{m}}(G_{1},-(t+h)\Phi,\delta) \geq P'_{Z_{m}}(G_{1},-t\Phi,\delta)-h(\beta+\varepsilon).
\]
Let $\delta \rightarrow 0$ and we can get
\[
P_{Z_{m}}(G_{1},-(t+h)\Phi) \geq P_{Z_{m}}(G_{1},-t\Phi)-h(\beta+\varepsilon).
\]
Take the supremum over all $m\geq 1$, and by the Proposition \ref{3.1p} we can get
\[
P_{Z}(G_{1},-(t+h)\Phi) \geq P_{Z}(G_{1},-t\Phi)-h(\beta+\varepsilon).
\]
Since $\varepsilon > 0$ is arbitrary, the left of inequality (\ref{7.1}) can be given.
This complete the proof of the inequality (\ref{7.1}).

(2) Note that the map $t \mapsto P_{Z}(G_{1},-t\Phi)$ is continuous and strictly decreasing by (1). First let $t=0$ and $h=\frac{h_{Z}(G_{1})}{\beta}$ on the left of inequality (\ref{7.1}), and then we have
\[
P_{Z}(G_{1},-\frac{h_{Z}(G_{1})}{\beta}\Phi)\geq P_{Z}(G_{1},0)-h_{Z}(G_{1})=0.
\]
Second let $t=0$ and $h=\frac{h_{Z}(G_{1})}{\alpha}$ on the right of inequality (\ref{7.1}), we can get
\[
P_{Z}(G_{1},-\frac{h_{Z}(G_{1})}{\alpha}\Phi) \leq P_{Z}(G_{1},0)-h_{Z}(G_{1})=0.
\]
Consequently, we can obtain the desired result by Intermediate Value Theorem.

(3) It follows from (2) immediately.
\end{proof}

\begin{lemma}\label{7.1l}
Let $f_{i}: X \longrightarrow X$ be as in the Theorem 4.2.
Then given any $x \in \mathcal{B}$ and $\varepsilon > 0$, there exist $\delta_{0}=\delta_{0}(\varepsilon)>0$,  $\eta=\eta(x,\varepsilon)>0$ such that for each $n \in \mathbb{N}$ and $0 < \delta < \delta_{0}$,
\begin{equation}\label{7.3}
B(x,\eta\delta e^{-n(\lambda_{n}(x)+\varepsilon)})\subset B_{n}(x,\delta) \subset B(x,\delta e^{-n(\lambda_{n}(x)-\varepsilon)}).
\end{equation}
\end{lemma}

\begin{proof}
Owing to the fact that $f_{i}$ is conformal with factor $a_{i}(x)>0$ for each $i \in \{0,1,2\cdots,k-1\}$, we have
\[
\lim_{y\rightarrow x}\frac{d(f_{i}(x),f_{i}(y))}{d(x,y)}=a_{i}(x).
\]
Since $a_{i}(x)>0$ everywhere, we can take logarithms of the above equation and get
\[
\lim_{y\rightarrow x}\bigg (\log d(f_{i}(x),f_{i}(y))-\log d(x,y)\bigg)=\log a_{i}(x).
\]
Then we can extend it to a continuous function $\zeta_{i}: X \times X\longrightarrow \mathbb{R}$
\begin{equation*}
\zeta_{i}(x,y)=\begin{cases}
\log d(f_{i}(x),f_{i}(y)-\log d(x,y)  \       \ x\neq y, \\
\log a_{i}(x) \       \ x=y.
\end{cases}
\end{equation*}
Similar to the proof of Lemma 6.1 in \cite{Xiao}, for any $\varepsilon >0$ we can also get
\begin{equation}\label{7.4}
d(f_{i}(x),f_{i}(y))e^{-(\log a_{i}(x)+\varepsilon)} <  d(x,y)<  d(f_{i}(x),f_{i}(y))e^{-(\log a_{i}(x)-\varepsilon)} ,
\end{equation}
whenever the middle quantity is less than $\delta$.

Now we prove the second half of (\ref{7.3}). Let $\Phi=\{\log a_{0},\cdots, \log a_{k-1}\}$ and for any $y \in B_{n}(x,\delta)$, we have $d(g(x),g(y)) < \delta$ for all $g \in G_{n}$.
Then for any $w=i_{1}i_{2}\cdots i_{n} \in F^{+}_{k}(n)$, repeated application of the second inequality in (\ref{7.4}) yields
\begin{align*}
 d(x,y)&< d(f_{i_{1}}(x),f_{i_{1}}(y)) e^{-(\log a_{i_{1}}(x)-\varepsilon)}\\
 &<d(f_{i_{2}i_{1}}(x),f_{i_{2}i_{1}}(y)) e^{-(\log a_{i_{2}}(f_{i_{1}}(x))-\varepsilon)} e^{-(\log a_{i_{1}}(x)-\varepsilon)}\\
 &<\cdots\\
 &<d(f_{\overline{w}}(x),f_{\overline{w}}(y)) e^{-(S_{w}\Phi(x)-|w|\varepsilon )}\\
 &< \delta e^{-|w|(\lambda_{w}(x)-\varepsilon )}.
\end{align*}
Since $w \in F^{+}_{k}(n)$ is arbitrary, we can obtain
\[
d(x,y)^{k^{n}} < \prod_{|w|=n} \delta e^{-n(\lambda_{w}(x)-\varepsilon )},
\]
and then we have
\[
d(x,y)< \left(\prod_{|w|=n} \delta e^{-n(\lambda_{w}(x)-\varepsilon )}\right)^{\frac{1}{k^n}} = \delta e^{-n(\lambda_{n}(x)-\varepsilon)}.
\]
Thus
\[
B_{n}(x,\delta) \subset B(x,\delta e^{-n(\lambda_{n}(x)-\varepsilon)}).
\]

Next we show the first inclusion in (\ref{7.3}). Note that for any fixed $w=i_{1}i_{2}\cdots i_{n} \in F^{+}_{k}(n)$, if $d(x,y)<\delta$, then by the first inequality in (\ref{7.4}) we get
\[
d(f_{i_1}(x),f_{i_1}(y))<d(x,y) e^{\log a_{i_{1}}(x)+\varepsilon}.
\]
Then if $d(x,y)<\delta e^{-(\log a_{i_{1}}(x)+\varepsilon)}$, we have $d(f_{i_1}(x),f_{i_1}(y))<\delta$ and so
\begin{align*}
d(f_{i_{2}i_{1}}(x),f_{i_{2}i_{1}}(y))&<d(f_{i_1}(x),f_{i_1}(y))e^{\log a_{i_{2}}(f_{i_{1}}x)+\varepsilon}\\
&<d(x,y)e^{\log a_{i_{2}}(f_{i_{1}}x)+\varepsilon} e^{\log a_{i_{1}}(x)+\varepsilon}.
\end{align*}
Using this method repeatedly, we can obtain that if
\[
d(x,y)<\delta e^{-|w'|(\lambda_{w'}(x)+\varepsilon)}
\]
for each $\overline{w'}\leq \overline{w}$, we have $d(f_{\overline{w'}}(x),f_{\overline{w'}}(y))<\delta$, namely, $y \in B_{w}(x, \delta)$.
Therefore
\begin{equation*}
B(x,\delta \min_{\overline{w'}\leq \overline{w}}e^{-|w'|(\lambda_{w'}(x)+\varepsilon)}) \subset B_{w}(x,\delta).
\end{equation*}
Because of $B_{n}(x,\delta)=\cap_{|w|=n} B_{w}(x,\delta)$, it is easy to get
\begin{equation}\label{7.5}
B(x,\delta \min_{\overline{w'}\leq \overline{w},|w|=n}e^{-|w'|(\lambda_{w'}(x)+\varepsilon)}) \subset B_{n}(x,\delta).
\end{equation}
Now we find what $\eta$ should be for any $\overline{w'}\leq \overline{w}$ and $|w|=n$, and we observe that
\begin{align*}
\frac{e^{-n(\lambda_{n}(x)+2\varepsilon)}}{e^{-|w'|(\lambda_{w'}(x)+\varepsilon)}}
&=\frac{e^{-\frac{1}{k^{n}}\Sigma_{|w|=n}S_{w}\Phi(x)-2n\varepsilon}}{e^{-S_{w'}\Phi(x)-|w'|\varepsilon}}\\
&=e^{-\big(\frac{1}{k^{n}}\Sigma_{|w|=n}S_{w}\Phi(x)-S_{w'}\Phi(x)+2n\varepsilon-|w'|\varepsilon\big)}\\
&\leq e^{-\big(\frac{1}{k^{n}}\Sigma_{|w|=n}S_{w}\Phi(x)-S_{w'}\Phi(x)+n\varepsilon\big)}.
\end{align*}
Because $x$ satisfies the tempered contraction condition, there exists $\eta=\eta(x,\varepsilon)>0$ such that for arbitrary $w \in F_{k}^{+}(n),~\overline{w'}\leq \overline{w}$,
\[
 \frac{1}{k^{n}}\sum_{|w|=n}S_{w}\Phi(x)-S_{w'}\Phi(x)+n\varepsilon > \log \eta ,
\]
 and hence
\[
e^{-\big(\frac{1}{k^{n}}\Sigma_{|w|=n}S_{w}\Phi(x)-S_{w'}\Phi(x)+n\varepsilon\big)}< \frac{1}{\eta}.
\]
Then for any $|w|=n,\ \overline{w'}\leq \overline{w}$, we have
\[
\eta e^{-n(\lambda_{n}(x)+2\varepsilon)}< e^{-|w'|(\lambda_{w'}(x)+\varepsilon)},
\]
and combining with (\ref{7.5})
\[
B(x,\delta \eta e^{-n(\lambda_{n}(x)+2\varepsilon)}) \subset B_{n}(x,\delta).
\]
Taking $\delta_{0}=\delta_{0}(\varepsilon/2)$ gives
\[
B(x,\eta\delta e^{-n(\lambda_{n}(x)+\varepsilon)}) \subset B_{n}(x,\delta) .
\]
\end{proof}

\begin{lemma}\label{7.2l}
Let $f_{i}: X\longrightarrow X$ satisfy the conditions of Theorem 4.2 and fix $Z \subset \mathcal{A}((\alpha,\infty))\cap \mathcal{B}$, where $0<\alpha<\infty$. Let $t^{*}$ be the unique real number with $P_{Z}(G_{1},-t^{*} \Phi )=0$, whose existence and uniqueness are guaranteed by Proposition \ref{7.1p}, where $\Phi=\{\log a_{0},\cdots,~\log a_{k-1}\}$. Then $dim_{H} Z = t^{*}$.
\end{lemma}

\begin{proof}
Firstly we show $dim_{H} Z \leq t^{*}$. Given $m \geq 1$, and let
\[
Z_{m}=\{x \in Z: \lambda_{n}(x)> \alpha~~\text{for all}~n\geq m\}
\]
and then it is obvious to examine that $Z=\bigcup_{m=1}^{\infty} Z_{m}$. For any $t > t^{*}$, we can get $P_{Z}(G_{1},-t \Phi)< 0$. Thus there exists $\varepsilon >0$ with $P_{Z}(G_{1},-t \Phi )<-t\varepsilon$ and by Lemma \ref{7.1l} there exists $\delta_{0}=\delta_{0}(\varepsilon)$ such that for all $x \in Z_{m},~0<\delta<\delta_{0}$, and $n \geq m$, we have
\begin{equation}\label{7.6}
diam B_{n} (x,\delta)\leq 2 \delta e^{-n(\lambda_{n}(x)-\varepsilon)} \leq 2 \delta e^{-n(\alpha-\varepsilon)}.
\end{equation}
For given $N\geq m$ and $0< \delta < \delta_{0}$, we have
\begin{align*}
M'(Z_{m},G_{1}, -t \Phi,-t\varepsilon,\delta,N)
&=\inf\limits_{\mathcal{G} }\left\{\sum\limits_{B_{n}(x,\delta)\in\mathcal{G}}\exp \bigg(-(-t\varepsilon) \cdot n-t  \frac{1}{k^{n}} \sum_{|w|=n}S_{w}\Phi(x)\bigg)\right\}\\
&=\inf\limits_{\mathcal{G}}\left\{\sum\limits_{B_{n}(x,\delta)\in\mathcal{G}}\exp \bigg(-t \cdot n(\lambda_{n}(x)-\varepsilon)\bigg)\right\}\\
&\geq \inf\limits_{\mathcal{G}}\left\{\sum\limits_{B_{n}(x,\delta)\in\mathcal{G}} (\frac{1}{2\delta} diam B_{n}(x,\delta))^{t}\right\} \\
&\geq \inf\limits_{\mathcal{P}(Z_{m},2\delta e^{-N(\alpha-\varepsilon)})} \left\{\sum_{U_{i} \in \mathcal{P}(Z_{m},2\delta e^{-N(\alpha-\varepsilon)})} (\frac{1}{2\delta} diam U_{i})^{t}\right\} \\
&=(2\delta)^{-t} \mathcal{H} (Z_{m},t,2\delta e^{-N(\alpha-\varepsilon)}), \\
\end{align*}
where $\mathcal{G} $ covers $Z_{m}$ satisfying for any $B_{n}(x,\delta) \in \mathcal{G}$, ~$n \geq N$ and $\mathcal{P}(Z_{m},2\delta e^{-N(\alpha-\varepsilon)})$ denotes the collection of open covers $\{ U_{i}\}$ of $Z_{m}$ for which $diam U_{i} < 2\delta e^{-N(\alpha-\varepsilon)}$ for all $i$.\\
Take the limit as $N \rightarrow \infty$ and we get
\begin{equation}\label{7.7}
m'(Z_{m},G_{1}, -t \Phi,-t\varepsilon,\delta)  \geq (2\delta)^{-t} \mathcal{H}(Z_{m},t).
\end{equation}
Moreover, we have
\[
-t\varepsilon > P_{Z}(G_{1},-t \Phi )\geq P_{Z_{m}}(G_{1},-t \Phi)=\lim_{\delta \rightarrow 0} P'_{Z_{m}}(G_{1},-t \Phi, \delta),
\]
and for sufficiently small $\delta >0$, we have $-t \varepsilon > P'_{Z_m}(G_{1},-t \Phi,\delta )$.
Hence $ \mathcal{H}(Z_{m},t)=0$ by (\ref{7.7}), which means $ dim_{H}(Z_{m}) \leq t$. Then taking the union over all $m$ gives $dim_{H}(Z) \leq  t $ for all $t>t^{*}$. Hence $dim_{H}(Z) \leq t^{*}$.

Next we show the opposite inequality, $dim_{H}(Z) \geq t^{*}$. Consider $t <t^{*}$, and then we manage to prove $dim_{H}(Z) \geq t$. Assume $t>0$, otherwise there is nothing left to prove. By the Proposition \ref{7.1p}, we know $t^{*}$ is the unique real number such that $P_{Z}(G_{1},-t^{*} \Phi)=0$. Since the pressure function $P_{Z}(G_{1},-t\Phi)$ with respect to $t$ is decreasing, we have $P_{Z}(G_{1},-t \Phi)>0$. Hence we can choose $\varepsilon >0$ satisfying
\[
P_{Z}(G_{1},-t\Phi )>t\varepsilon > 0.
\]
Let $\delta_{0}=\delta_{0}(\varepsilon)$ be given in Lemma \ref{7.1l}. For given $n\in \mathbb{N},~m\geq 1$, consider the set
\[
Z_{m}=\{x \in Z: (\ref{7.3})~holds~with~\eta=e^{-m}~, ~0<\delta<\delta_{0}\}.
\]
It is easy to observe that $Z=\bigcup\limits_{m=1}^{\infty} Z_{m}$ and hence $P_{Z}(G_{1},-t\Phi)=\sup\limits_{m\geq 1} P_{Z_m}(G_{1},-t \Phi)$. Then there exists $m\in \mathbb{N}$ such that $P_{Z_m}(G_{1},-t\Phi)> t\varepsilon$, and we can choose sufficiently small $\delta$ with $0<\delta<\delta_{0}$ such that
\begin{equation}\label{7.8}
P'_{Z_m}(G_{1},-t\Phi,\delta)> t\varepsilon .
\end{equation}
Let $\beta=\max_{i} \{\sup_{x \in X}\log a_{i}(x)\}<\infty $. For any $n \in \mathbb{N},~x\in X$, denote $s_{n}(x)=e^{-m} \delta e^{-n(\lambda_{n}(x)+\varepsilon)}$,
and observe that
\[
\frac{s_{n}(x)}{s_{n+1}(x)}=\frac{e^{-n(\lambda_{n}(x)+\varepsilon)}}{e^{-(n+1)(\lambda_{n+1}(x)+\varepsilon)}}
=e^{(n+1)\lambda_{n+1}(x)-n\lambda_{n}(x)+\varepsilon}
\leq e^{\beta+\varepsilon}.
\]
Furthermore, for given $x \in Z_{m}$ and $r>0$ enough small, there exists $n=n(x,r) \in \mathbb{N}$ such that
\begin{equation}\label{7.9}
s_{n}(x)e^{-(\beta+\varepsilon)}\leq s_{n+1}(x) \leq r \leq s_{n}(x)=e^{-m} \delta e^{-n(\lambda_{n}(x)+\varepsilon)}.
\end{equation}
For this value $n$, by Lemma \ref{7.1l} we have
\[
B(x,r)\subset B_{n}(x,\delta).
\]
Then given any $\{B(x_{i},r_{i})\}$ such that $Z_{m} \subset \bigcup B(x_{i},r_{i})$, we can get $Z_{m} \subset \bigcup  B_{n_{i}}(x_{i},\delta)$, where $~n_{i}=n_{i}(x_{i},r_{i})$ satisfies (\ref{7.9}).

Moreover, for every $n \in \mathbb{N}$ and $x \in X$, we have $\lambda_{n}(x) \leq \beta$ and thus $s_{n}(x) \geq \delta e^{-\big(m+n(\beta+\varepsilon)\big)}$.
From (\ref{7.9}), it yields that for $n=n(x,r)$ we have
\[
\delta e^{-\big(m+(n+1)(\beta+\varepsilon)\big)} \leq r,
\]
and then
\[
n \geq \frac{-\log r +\log \delta -m}{\beta+\varepsilon}-1.
\]
Denote $N:=N(r,\delta)=\frac{-\log r +\log \delta -m}{\beta+\varepsilon}-1$ and note that for every fixed $\delta>0$, we have
$\lim_{r\rightarrow 0}N(r,\delta)=\infty$.

Therefore, for all $r>0,~0<\delta <\delta_{0}$, by applying (\ref{7.9}) we can get
\begin{align*}
\mathcal{H}^{b}(Z_{m},t,r)
&=\inf_{\mathcal{P}^{b}(Z_{m},r)}\left\{\sum\limits_{B(x_{i},r_{i})\in \mathcal{P}^{b}(Z_{m},r)} (2r_{i})^{t}\right\}\\
&\geq \inf\limits_{\mathcal{G}' }\left\{\sum\limits_{B_{n_{i}}(x_{i},\delta)\in\mathcal{G}'}\bigg(2 e^{-(\beta+\varepsilon)}  s_{n_{i}}({x_{i}})\bigg)^{t}\right\}\\
&=(2\delta)^{t}e^{-t(m+\beta+\varepsilon)} \inf\limits_{\mathcal{G}'}\left\{\sum\limits_{B_{n_{i}}(x_{i},\delta)\in\mathcal{G}'}  e^{-tn_{i} (\lambda_{n_{i}}(x_{i})+\varepsilon)}\right\} \\
&=(2\delta)^{t}e^{-t(m+\beta+\varepsilon)}  \inf\limits_{\mathcal{G}'}\left\{\sum\limits_{B_{n_{i}}(x_{i},\delta)\in\mathcal{G}'}  \exp \bigg(-t\varepsilon n_{i}-t \frac{1}{k^{n_{i}}}\sum_{|w|=n_{i}} S_{w}\Phi(x_{i})\bigg)\right\} \\
&\geq (2\delta)^{t}e^{-t(m+\beta+\varepsilon)}  \inf\limits_{\mathcal{G} }\left\{\sum\limits_{B_{n}(x,\delta)\in\mathcal{G}}  \exp \bigg(-t\varepsilon n-t \frac{1}{k^{n}}\sum_{|w|=n} S_{w}\Phi(x)\bigg)\right\} \\
&=(2\delta)^{t}e^{-t(m+\beta+\varepsilon)} M'(Z_{m},G_{1},-t\Phi,t\varepsilon,\delta,N) , \\
\end{align*}
where $\mathcal{P}^{b}(Z_{m},r)$ denotes the collection of countable open balls covers $\{B(x_{i},r_{i})\}$ of $Z_{m}$ for which $r_{i} < r$ for all $i$,
$\mathcal{G}'$ denotes the collection of countable Bowen balls covers $\{B_{n_{i}}(x_{i},\delta)\}_{i=1}^{\infty}$ of $Z_{m}$ with for any $B_{n_{i}}(x_{i},\delta) \in \mathcal{G}', n_{i}\geq N$, ~$B(x_{i},r_{i}) \subset B_{n_{i}}(x_{i},\delta)$
and $\mathcal{G}$ denotes the collection of countable Bowen balls covers $\{B_{n}(x,\delta)\}$ of $Z_{m}$ with for any $B_{n}(x,\delta) \in \mathcal{G}', n\geq N$.\\
Consequently, we can get
\[
\mathcal{H}^{b}(Z_{m},t,r)\geq (2\delta)^{t}e^{-t(m+\beta+\varepsilon)} M'(Z_{m},G_{1},-t\Phi,t\varepsilon,\delta,N).
\]
Taking $r\rightarrow 0$, it is easy to see that the quantity on the right goes to $\infty$ by (\ref{7.8}), and thus we have $\mathcal{H}^{b}(Z_{m},t)=\infty$.
Therefore,
\[
dim_{H}Z \geq dim_{H}Z_{m}\geq t,
\]
and since $t< t^{*}$ is arbitrary, this establishes the Lemma.
\end{proof}

\begin{proof}[Proof of Theorem 4.1]
 Choose a decreasing  positive numbers sequence $\{\alpha_{m}\}$ with converging to $0$, and set $Z_{m}= \mathcal{A}((\alpha_{m},\infty))\bigcap Z$, then Lemma \ref{7.2l} applies to $Z_{m}$ and we can get $Z=\bigcup_{m=1}^{\infty} Z_{m}$. Let $t_{m}$ be the unique real number with
\[
P_{Z_{m}}(G_{1},-t_{m} \Phi )=0
\]
for every $m$ and whose existence and uniqueness are guaranteed by Proposition \ref{7.1p}. Hence, by Lemma \ref{7.2l} we obtain
\[
dim_{H} Z_{m}=t_{m}.
\]
Denote $t^{*}=\sup_{m} t_{m}$, then $dim_{H} Z=t^{*}$. Then it is left to prove
\begin{equation}\label{7.10}
t^{*}=\sup\{ t\geq 0: P_{Z}(G_{1},-t \Phi)>0\}.
\end{equation}
Considering $t\geq 0$, we have
\[
P_{Z}(G_{1},-t \Phi )=\sup_{m}P_{Z_{m}}(G_{1},-t \Phi ).
\]

 On the one hand, for any $t < t^{*}$, there exists $t_{m}$ with $t<t_{m}$ and hence $P_{Z_{m}}(G_{1},-t \Phi )>0$. Then $t \in \{ t\geq 0: P_{Z}(G_{1},-t \Phi )>0\}$ and thus $t \leq \sup\{ t\geq 0: P_{Z}(G_{1},-t\Phi )>0\}$. Therefore, $t^{*} \leq \sup\{ t\geq 0: P_{Z}(G_{1},-t\Phi )>0\}$.

 On the other hand, for arbitrary $t < \sup\{ t\geq 0: P_{Z}(G_{1},-t \Phi)>0\}$, there exists $t_{j}>t$ with $P_{Z}(G_{1},-t_{j}\Phi )>0$. Then there exists $Z_{m}$ such that $P_{Z_{m}}(G_{1},-t_{j} \Phi )>0$ and thus $t_{j} < t_{m}$. It follows that $t<t_{j}<t_{m}<t^{*}$.
So  $\sup\{ t\geq 0: P_{Z}(G_{1},-t\Phi )>0\} \leq t^{*}$.
This establishes (\ref{7.10}).

Finally, it can be yielded from (\ref{7.10}) and continuity of $t \mapsto P_{Z}(G_{1}, -t \Phi)$ that $P_{Z}(G_{1}, -t^{*} \Phi)=0$. If $Z \subset \mathcal{A}((\alpha,\infty))$ for some $\alpha >0$, then Proposition \ref{7.1p} guarantees that $t^{*}$ is in fact the unique root of Bowen's equation.
\end{proof}

{\bf Acknowledgement.}  The work was supported by National Natural Science Foundation of China (grant no.11771149, 11671149) and Guangdong Natural Science Foundation 2018B0303110005.

\bibliographystyle{amsplain}

\end{document}